\documentclass[A4paper,12pt]{article}
\usepackage{amsmath,amsthm,amssymb}
\usepackage[margin=1.2in]{geometry}

\usepackage{color}
\usepackage{array}
\usepackage{comment}

\newtheorem{thm}{Theorem}[section]

\newtheorem{lem}[thm]{Lemma}

\newtheorem{exam}[thm]{Example}

\newtheorem{Cor}[thm]{Corollary}

\numberwithin{equation}{section}

\newcommand{\vect}[1]{\mbox{\boldmath ${#1}$}}
\newcommand{\rev}[1]{{#1}}

\makeatletter

\def\authoranddepartment#1{\def\@department{#1}}
\def\authors#1{\def\@authors{#1}}
\def\abstract#1{\def\@abstract{{\bf Abstract: }#1}}

\def\@maketitle{
\begin{center}
{\Large\bf \@title \par}
\end{center}
\begin{center}
{\large \@authors \par}
\end{center}
\begin{center}
{\@department \par}
\end{center}
{\@abstract}
}

\makeatother

\title{A Method of Construction of Pairwise Additive Balanced Incomplete Block Designs}
\authors{Kazuki Matsubara$^1$ and Satoru Kadowaki$^2$\footnote{Deceased, 2022}}
 \authoranddepartment{ ${}^1$Faculty of Education, Saitama University, 255 Shimo-Okubo, Sakura-ku, Saitama 338-8570, Japan \\
 ${}^2$Department of Science, National Institute of Technology, Matsue College, 14-4 Nishiikumacho, Matsue, Shimane 690-8518, Japan\\
 \textbf{Correspondence:} Kazuki Matsubara (additivestructure@gmail.com)}
 
 \abstract{
 The existence of sets of balanced incomplete block (BIB) designs with pairwise additivity, called pairwise additive BIB designs, has been studied through direct and recursive constructions in the literature.
This paper presents a new method of constructing such designs from affine resolvable semi-regular group divisible (ARSRGD) designs.
Since many known ARSRGD designs are obtained from generalized Hadamard matrices, we also describe the corresponding construction of pairwise additive BIB designs from generalized Hadamard matrices.
Using the construction together with known existence results for generalized Hadamard matrices, we obtain several new infinite series of pairwise additive BIB designs, including designs with parameter sets whose existence was previously unknown.
}

\begin{document}

\maketitle

\vspace{2mm}

\vspace{2mm}

\vspace{2mm}
\noindent  
\section{Introduction}\label{sec:introduction}
A \textit{block design} is a system $(V, \mathcal B)$, denoted by BD($v,b,r,k$), with $v$ points ($|V| = v$) and $b$ blocks ($|\mathcal B| = b$) containing $k$ different points 
\rev{with $2\le k\le v$}, and each point appearing in $r$ different blocks.
A $(V, \mathcal B)$ is called a \textit{balanced incomplete block (BIB) design}, denoted by BIBD$(v, b, r, k, \lambda)$ if any two different points appear in exactly $\lambda$ blocks.
Since the parameters satisfy $vr=bk$ and $\lambda(v-1) = r(k-1)$,
the parameters $b$ and $r$ are determined by $v$, $k$, and $\lambda$ (cf.\ \cite{R1988}). 
Thus, we use $v$, $k$, and $\lambda$ as the essential parameters of a BIB design throughout this paper.
An \textit{incidence matrix} $\vect N = (n_{ij})$ is a $v \times b$ matrix such that $n_{ij} = 1$ or $0$ for all $i$ ($=1,2,\ldots , v$) and $j$ ($= 1,2,\ldots , b$), according as the $i$-th point occurs in the $j$-th block or otherwise. 
\rev{In this paper, we allow $k=v$. Thus, a block design whose incidence matrix is $\vect J$ is regarded as a BIB design, where $\vect J$ denotes the all-one matrix.}

Let  $\vect N_h = (a_{ij}^{(h)})$ ($1\le h\le \ell$) be $\ell$ incidence matrices of BIB designs with the same parameters such that at most one of $a_{ij}^{(h_1)}$ and $a_{ij}^{(h_2)}$ is equal to $1$, i.e., $a_{ij}^{(h_1)}a_{ij}^{(h_2)}=0$, for any $i$, $j$ and any distinct $h_1, h_2 \in \{1,2, \ldots, \ell\}$. 
A set $\{\vect N_1, \vect N_2, \ldots, \vect N_{\ell}\}$ is called a \textit{pairwise additive BIB design}, denoted by PAB($\ell$; $v$, $k$, $\lambda$), if $\vect N_h$ is the incidence matrix of a BIB design $(V,\mathcal B_h)$, BIBD($v,b,r,k,\lambda$), for each $h\in \{1,2,\ldots, \ell\}$ and $\vect N_{h_1} + \vect N_{h_2}$ is the incidence matrix of a BIB design with parameters
\begin{eqnarray*}\label{eq:parameters}
v^\ast = v,\ b^\ast = b,\ r^\ast = 2r,\ k^\ast = 2k,\ \lambda^\ast = 2r(2k-1)/(v-1)
\end{eqnarray*}
for any distinct  $ h_1,  h_2 \in \{1,2,\ldots, \ell\}$.

Let $s = v/k$, where $s$ need not be an integer, unlike other parameters.
Clearly, $\ell \le s$ and the existence of a PAB($\ell; v,k,\lambda$) implies the existence of a PAB($\ell'; v,k,\lambda$) with $\ell'<\ell$.
This paper focuses on the pairwise additive BIB designs with $\ell = s=v/k$, which are called \textit{full pairwise additive BIB designs}.
For convenience and to distinguish this full case from general PABs, a PAB$(\ell; v, k, \lambda)$ with $\ell =v/k$ is also denoted by full PAB$(v,k,\lambda)$.
These full pairwise additive BIB designs were called additive BIB designs in \cite{SMMKK2007}.
Obviously, $\sum_{h=1}^s \vect N_h$ forms an all-one matrix for any full pairwise additive BIB design.

The concept of pairwise additivity was introduced in \cite{MSMKK2006}.
Note that, in the literature, a pairwise additive BIB design is defined as a set of BIB designs, i.e., $(V,\mathcal B_1), (V,\mathcal B_2),\ldots, (V,\mathcal B_{\ell})$, together with a suitable ordering of their blocks.
In this paper, we instead define it as a set of the incidence matrices, so that the correspondence between blocks is explicitly specified.
Existence results, construction methods, and applications have been discussed in \cite{SMMKK2007}.
It is also known that pairwise additive BIB designs are closely related to other combinatorial structures, e.g., compatible minimal partitions of triple systems and  balanced nested designs given in \cite{CR1999} and \cite{FKKMS2002}, respectively.
Several construction methods of pairwise additive BIB designs have been discussed mainly for $\ell=2,3$, $k=2,3$, and $\lambda=1$ in \cite{MK2013-1,MK2013-2,MK2014,MK2015,MHK2015,MSMKK2006}.
\rev{In contrast, existence results for full PABs, where $\ell=v/k$, are much more limited.}
For the case $\ell=v/k$, \cite{SMMKK2007} and \cite{SKJ2008} show the existence of a full PAB$(2^n,2,1)$ and a full PAB$(3^n,3,1)$ for any $n\ge 2$, respectively.
Furthermore, some classes of full pairwise additive BIB designs have been constructed from BIB designs with the resolvability introduced in Section \ref{sec:preliminary}.

The lower bound for $\lambda$ in a PAB($\ell; v,k,\lambda$) is given by
\begin{eqnarray}
 \lambda \ge 
 \left\{
  \begin{array}{cl}
   (k-1)/2 & \mbox{if}\ k\equiv 1\ (\mbox{mod}\ 2)\\
   k-1 & \mbox{otherwise}
  \end{array}.
 \right.\label{eq:lambda}
\end{eqnarray} 
In \cite{SMMKK2007}, the spectrum of full pairwise additive BIB designs within the scope of parameters of full pairwise additive BIB designs with $s\ge 3, v\le 100, k\ge 2, r\le 20$, and $\lambda <k$ is given in the Appendix, together with the list of the $20$ sets of parameters for each of which the existence of a full PAB$(v,k,\lambda)$ was unknown.
Of those $20$ sets, a full PAB$(10,2,1)$ and a full PAB$(27,3,1)$ have been constructed in \cite{MK2013-1} and \cite{SKJ2008}, respectively.
\rev{To the best of the authors' knowledge, no subsequent constructions have been reported for the remaining $18$ sets of parameters.
These parameter sets are listed in Table~\ref{table}.}
\begin{table}[hbtp]
\centering
 $\begin{array}{c|cccccccccccccccccc}
 \mbox{No.} &1&2&3&4&5&6&7&8&9&10
 &11&12&13&14&15&16&17&18\\
 \hline\hline
  s & 3 & 3 & 3 & 3 & 4  & 5 & 5 & 5 & 6 
     & 6 & 7 & 7 & 7 & 7 & 9 & 10 & 11 & 13 \\
  \hline
  v & 15 & 21 & 33 & 39 & 20 & 15 & 20 & 35 & 12 
     & 18 & 14 & 21 & 21 & 35 & 18 & 20 & 33 & 39  \\
  \hline
  k & 5 & 7 & 11 & 13 & 5 & 3 & 4 & 7 & 2 
     & 3 & 2 & 3 & 3 & 5 & 2 & 2 & 3 & 3\\
  \hline
  \lambda 
     & 4 & 3 & 5 & 6 & 4 & 1 & 3 & 3 & 1 & 2 
     & 1 & 1 & 2 & 2 & 1 & 1 & 1 & 1
 \end{array}$
 \caption{The list of unknown parameters of full pairwise additive BIB designs.}
 \label{table}
\end{table}

The purpose of this paper is to provide a new construction method for full pairwise additive BIB designs and to obtain several new infinite series of full PABs, thereby establishing the existence of full PABs for Nos.\,7, 11, and 15 in Table \ref{table}.
Section \ref{sec:preliminary} reviews some properties of  PABs and some results on affine resolvable semi-regular group divisible (ARSRGD) designs in the literature.
In Section \ref{sec:fromSRGD}, as the main result, we provide a new construction method for full PABs with $\lambda = k-1$ from ARSRGD designs.
In Section \ref{sec:GHM}, we consider generalized Hadamard matrices as a source of full PABs based on the construction in Section \ref{sec:fromSRGD}.
First, known existence results for generalized Hadamard matrices are reviewed. 
Then, the construction is applied to obtain several new infinite series of full PABs with $\lambda=k-1$, which satisfy the lower bound in (\ref{eq:lambda}).
In Section \ref{sec:conclusion}, the results of this paper are summarized.
Finally, some contributions to other combinatorial structures are also described as future work.

\section{Preliminaries}\label{sec:preliminary}
In this section, we introduce the basic properties of pairwise additivity and some results on affine resolvable semi-regular group divisible designs.

For a BIBD($v,b,r,k,\lambda$), the incidence matrix $\vect N = (n_{ij})$ satisfies $\sum_{i=1}^v n_{ij}=k, \sum_{j=1}^b n_{ij}=r$, and
\begin{align}\label{eq:N}
 \vect N \vect N^T = (r-\lambda)\vect I_v + \lambda \vect J_v,
\end{align}
where $\vect I_v$ and $\vect J_v$ are the identity matrix and all-one matrix of order $v$, respectively.
A set of $\ell$ incidence matrices $\vect N_1, \vect N_2, \ldots, \vect N_{\ell}$ of a PAB($\ell; v,k,\lambda$) satisfies the following conditions:
\begin{align}\label{eq:N_1+N_2}
 \left(
 \vect N_{h_1} + \vect N_{h_2}
 \right)
 \left(
 \vect N_{h_1} + \vect N_{h_2}
 \right)^T
 = (2r-\lambda^\ast)\vect I_v + \lambda^\ast \vect J_v
 \end{align}
for any distinct $h_1, h_2\in \{1,2,\ldots,\ell\}$, where $\lambda^\ast = 2r(2k-1)/(v-1) = 2\lambda(2k-1)/(k-1)$.
Combining (\ref{eq:N}) and (\ref{eq:N_1+N_2}) leads to
\begin{align}\label{eq:N_1N_2}
 \vect N_{h_1}\vect N_{h_2}^T + \vect N_{h_2}\vect N_{h_1}^T 
 = (\lambda^\ast -2\lambda)\left(\vect J_v -  \vect I_v\right).
\end{align}
Assuming (\ref{eq:N}), condition (\ref{eq:N_1N_2}) is equivalent to (\ref{eq:N_1+N_2}). Hence, (\ref{eq:N}) and (\ref{eq:N_1N_2}) together are equivalent to the definition of pairwise additivity.

A BD($v,b,r,k$) is called a \textit{group divisible} (GD) design with parameters 
$v=mn,$ $b,$ $r,$ $k,$ $\lambda_1$, and $\lambda_2$, denoted by GD$(m,n,k,\lambda_1,\lambda_2)$,
if the $v$ ($=mn$) points are divided into $m$ groups of $n$ points such that any two points in the same group occur together in exactly $\lambda_1$ blocks, whereas any two points from different groups occur together in exactly 
$\lambda_2$ blocks.
Then the parameters satisfy $vr=bk$ and $r(k-1) = (n-1)\lambda_1 + n(m-1)\lambda_2$.
The GD designs are further classified into three subclasses: 
singular if $r-\lambda_1=0$; 
semi-regular (SR) if $r-\lambda_1>0$ and $rk-v\lambda_2=0$; 
and regular if $r-\lambda_1>0$ and $rk-v\lambda_2>0$ (cf. \cite{R1988}).
In particular, the following property is shown in the literature.

\begin{lem}[{\cite[Theorem 8.5.6]{R1988}}]
\label{lem:semi-regular}
 For a semi-regular GD design, $k$ is divisible by $m$.
 If $k=cm$, then every block contains $c$ symbols from each group.
\end{lem}

Next, we introduce the resolvability of blocks closely related to full PABs.
A BD($v, b, r, k$) is said to be \textit{resolvable} if the $b$ blocks can be grouped into $r$ resolution sets of $b/r(=v/k)$ blocks such that every point occurs exactly once in each resolution set. 
Furthermore, a resolvable BD is said to be \textit{affine resolvable} (AR) if every two blocks belonging to different resolution sets intersect in the same number, say $q$ ($=k\lambda_2/r$), of points (see \cite{KK2018}).

For the construction of full PABs, 
\rev{one of the constructions in \cite{SMMKK2007} is based on affine resolvable BIB designs, in which blocks of the input design are replicated.
In contrast, our construction starts from an affine resolvable GD design and extracts submatrices from its incidence matrix to form a full PAB.
Thus, although both constructions involve affine resolvable designs, the underlying designs and the construction procedures are different.}
We mainly use affine resolvable GD designs with $k=m$, $\lambda_1=0$, and $m,n\ge 2$.
Note that no affine resolvable regular GD design exists \cite{K2008}.
Moreover, $\lambda_1=0$ implies that the design is not singular.
Hence, the affine resolvable GD designs used in this paper are semi-regular, and their existence is equivalent to the existence of a class of orthogonal arrays (see \cite[Theorem 8.5.7]{R1988}).
Although it is not entirely known for which parameters affine resolvable semi-regular group divisible (ARSRGD) designs exist, the ARSRGD designs can be characterized as follows.

\begin{lem}[{\cite[Corollary 8.5.10.1]{R1988}}]
\label{lem:affine}
 A resolvable SRGD design with parameters $v=mn, b, r, k, \lambda_1, \lambda_2, m, n$ is affine resolvable if and only if (a) $b=v-m+r$ and (b) $k^2/n$ is an integer.
\end{lem}

Let $\vect N = (a_{ij})$ be an ARSRGD design with $\lambda_1=0$ and $k=m$ such that $V=\{x_{fh}\mid 1\le f\le m, 1\le h\le n\}$ is a set of $mn$ points and the $i$th row of its incidence matrix  corresponds to the point $x_{fh}$ with $i = (f-1)n + h$.
Then $\vect N$ is called \textit{normalized} if the $((f-1)n+1)$-th, $((f-1)n+2)$-th, \ldots, $((f-1)n+n)$-th rows correspond to the points in the same group for each $f=1,2,\ldots,m$, and for the first $n$ columns
\begin{eqnarray}\label{normalize}
 a_{ij} = 
 \left\{
 \begin{array}{cc}
  1 & (i\equiv j, \mbox{mod}\ n)\\
  0 & (i\not\equiv j, \mbox{mod}\ n)
 \end{array}
 \right.
\end{eqnarray}
with $1\le i\le mn$ and $1\le j\le n$.
Without loss of generality, we can easily obtain a normalized incidence matrix from any ARSRGD design with $\lambda_1=0$ and $k=m$ by applying row and column permutations.
Throughout this paper, we assume that the incidence matrix of any ARSRGD design is normalized.

Note that, in any ARSRGD$(m,n,k,\lambda_1,\lambda_2)$, Lemma \ref{lem:semi-regular} with $\lambda_1=0$ gives $k=m$.
Moreover, $k=m$ and $rk-v\lambda_2=0$ \rev{imply} $r=n\lambda_2$.
Lemma \ref{lem:affine} gives $b=v-m+r=m(n-1)+n\lambda_2$.
On the other hand, $vr=bk$ implies $b=mnr/m=n^2\lambda_2$.
Thus, from $m(n-1)+n\lambda_2 = n^2\lambda_2$ with $n\ge 2$, it follows that $m=n\lambda_2$.
Therefore, when $n\ge 2$, we have
\begin{eqnarray*}
k=m=n\lambda_2 = r, \ v=mn=n^2\lambda_2=b
\end{eqnarray*}
 in any ARSRGD$(m,n,k,\lambda_1=0,\lambda_2)$.
Moreover, since $q=k\lambda_2/r = \lambda_2$, the transpose of the incidence matrix $\vect N$ of any ARSRGD design is also the incidence matrix of an ARSRGD design with the same parameters $(m,n,k,\lambda_1=0,\lambda_2)$.
Since the design with such \rev{a} dual property is called \textit{symmetric}, 
affine resolvable SRGD designs with $\lambda_1=0$ can be viewed as a special class of symmetric GD designs discussed in \cite{KS2018, SS2023}.

\section{Construction from ARSRGD designs}
\label{sec:fromSRGD}
This section presents the main construction of the paper, which provides a method for constructing full PABs from ARSRGD designs.
Although the construction is simple, it yields many infinite families of full PABs.

\begin{thm}\label{thm:main}
 The existence of an ARSRGD design with parameters 
\begin{eqnarray}
v=b=mn,\ r= k=m,\ \lambda_1=0,\ \rev{\lambda_2\ge 2} \label{ARSRGD}
\end{eqnarray}
implies the existence of a full PAB with parameters
\begin{eqnarray}
v^\ast=m=n\lambda_2,\ b^\ast = (m-1)n,\ r^\ast = m-1,\ k^\ast = \lambda_2,\ \lambda^\ast = \lambda_2-1. \label{para}
\end{eqnarray}
\end{thm}

\begin{proof}
  Let $\vect N=(a_{ij})$ ($1\le i\le mn, 1\le j\le mn$) be a normalized incidence matrix of an ARSRGD design with the parameters (\ref{ARSRGD}), and let $\vect x_{fh}$ be a row vector of length $mn$ in the $((f-1)n+h)$-th row of $\vect N$, where $1\le f\le m$ and $1\le h\le n$.
 For each $h$ ($1\le h\le n$), define the $m\times mn$ matrices $\vect N_{h}$ by
  \begin{eqnarray*}
   \vect N_h = 
   \begin{pmatrix}
    \vect x_{1h}\\
    \vect x_{2h}\\
    \vdots\\
    \vect x_{mh}
   \end{pmatrix}.
  \end{eqnarray*}
 Let $\vect N_{h}^\ast$ be the $m\times (m-1)n$ matrix obtained from $\vect N_{h}$ by removing its first $n$ columns.
 We show that $\{\vect N_{1}^\ast, \vect N_{2}^\ast, \ldots, \vect N_{n}^\ast\}$ forms 
 the required full PAB$(v^\ast, k^\ast, \lambda^\ast)$.

 First, we show that each $\vect N_{h}^\ast$ is an incidence matrix of \rev{a} BIBD($m,(m-1)n, m-1, \lambda_2, \lambda_2-1$).
 By (\ref{normalize}), in the first $n$ columns of $\vect N_h$, every entry of the $h$-th column of $\vect N_h$ is $1$ and every entry in the $h'$-th ($1\le h'<h, h<h'\le n$) column is $0$.
 Moreover, the given design satisfies
 \begin{eqnarray}\label{eq:innerproduct}
 \vect x_{fh} \cdot \vect x_{f'h'}
 =
 \left\{
 \begin{array}{cc}
  0 & (f=f')\\
  \lambda_2 & (f\neq f')
 \end{array}
 \right.,
\end{eqnarray}
 where $1\le f,f'\le m$, $1\le h,h'\le n$, and $(f,h)\neq (f',h')$.
 Hence, by removing the first $n$ columns, the inner product of two distinct rows of $\vect N_h^\ast$ is $\lambda_2-1$. 
 Therefore,
\begin{eqnarray*}\label{eq:main}
 \vect N_h^\ast (\vect N_h^\ast)^T = (m-\lambda_2)\vect I_m + (\lambda_2-1) \vect J_m.
\end{eqnarray*}
Thus, the parameters $v^\ast, b^\ast, r^\ast, \lambda^\ast$ coincide with those in (\ref{para}).
 Moreover, the affine resolvability of the given ARSRGD design implies that the inner product of any two columns of $\vect N$  belonging to different resolution sets is $\lambda_2$.
Hence, in $\vect N_{h}$, the inner product of the $h$-th and the $j$-th ($n+1\le j\le mn$) columns is also $\lambda_2$, that is, each column of $\vect N_{h}^\ast$ contains exactly $\lambda_2$ ones.  
Since the block size $k^\ast$ of each $\vect N_{h}^\ast$ is constant $\lambda_2$, each $\vect N_{h}^\ast$ is an incidence matrix of \rev{a} BIBD($m,(m-1)n, m-1, \lambda_2, \lambda_2-1$).

 Since every point occurs exactly once in each resolution set, the resolvability of the ARSRGD design with $\lambda_1=0$ implies that $\sum_{h=1}^{n}\vect N_h^\ast = \vect J_{m\times (m-1)n}$ holds.
 Finally, it is sufficient to prove (\ref{eq:N_1N_2}).
 Since $\lambda_1=0$, the diagonal entries of $\vect N_{h_1}^\ast (\vect N_{h_2}^\ast)^T$ are all zero.
 On the other hand, by (\ref{eq:innerproduct}), all off-diagonal entries of $\vect N_{h_1} \vect N_{h_2}^T$ are $\lambda_2$.
 Furthermore, by the normalization (\ref{normalize}), the deleted first $n$ columns contribute to the entries of $\vect N_h (\vect N _h)^T$, but not to those of $\vect N_{h_1} (\vect N_{h_2})^T$ for $h_1\neq h_2$.
Hence, the off-diagonal entries of $\vect N_{h_1}^\ast (\vect N_{h_2}^\ast)^T$ are the same as those of $\vect N_{h_1} \vect N_{h_2}^T$.
 Thus, we have
 \begin{eqnarray}
  \vect N_{h_1}^\ast  (\vect N_{h_2}^\ast)^T = \lambda_2(\vect J_m-\vect I_m)
 \end{eqnarray}
 and 
 \begin{eqnarray*}
 \vect N_{h_1}^\ast  (\vect N_{h_2}^\ast)^T + \vect N_{h_2}^\ast  (\vect N_{h_1}^\ast)^T = 2\lambda_2(\vect J_m-\vect I_m)
\end{eqnarray*}
for $h_1 \neq h_2$.
The proof is complete.
\end{proof}

\section{Full PABs from generalized Hadamard matrices}\label{sec:GHM}

Generalized Hadamard matrices provide a useful source of ARSRGD designs and are particularly suitable for obtaining concrete existence results for full PABs.

A \textit{difference matrix}, denoted by DM($n,k,\lambda$), based on a finite abelian group $G$ of order $n$ is a $k\times n\lambda$ matrix $\vect A=(a_{ij})$ with entries from $G$ such that the multi-set $\{a_{i_1j} - a_{i_2j} \mid j=1,\ldots, n\lambda\}$ contains every element of $G$ precisely $\lambda$ times for $1\le i_1<i_2\le k$.
Note that a DM($n, k,\lambda$) is also called a difference scheme (e.g., \cite{HSS1999}).

A \textit{generalized Hadamard matrix}, denoted by GH($n,\lambda$), is a DM($n, n\lambda, \lambda$), i.e., an $n\lambda \times n\lambda$ matrix.
Construction methods of ARSRGD designs from difference matrices (or generalized Hadamard matrices) have been discussed in the literature (e.g., \cite{KK2018,KS2018,SS2023}).
It is called \textit{normalized} if all the symbols in the first row and first column are the identity element of the group.
Moreover, GH($n,\lambda$) has the following property.

\begin{lem}[\cite{B1988}]\label{lem:GHproperty}
 The transpose of a generalized Hadamard matrix is again a generalized Hadamard matrix.
\end{lem}
This result does not generalize to non-abelian groups, as shown by \cite{CL2009}.
Since our construction requires the property in Lemma \ref{lem:GHproperty}, throughout this paper we define difference matrices over abelian groups.
Many classes of generalized Hadamard matrices on an abelian group $G$ are known in the literature 
(see \cite{C2007,HSS1999,L2007}).
For example, the following existence results are known.
\rev{
\begin{lem}\label{lem:GH}
 For any prime power $q\ge 2$, there exist a GH$(q,2)$ and a GH$(q,4)$ over the elementary abelian group $\mbox{EA}(q)=(\mbox{GF}(q),+)$.
\end{lem}
}

\rev{
Indeed, the existence of GH$(q,2)$ over EA$(q)$ follows from Beth et al.~\cite[Theorem 3.14]{BJL1999}, while a construction of GH$(q,4)$ over EA$(q)$ was given by Dawson~\cite{D1985}.
In particular, there exist a GH$(5,4)$ over EA$(5)$, a GH$(7,2)$ over EA$(7)$, and a GH$(9,2)$ over EA$(9)$.
}

Although the following construction is equivalent to applying Theorem~\ref{thm:main} to the ARSRGD design obtained from a generalized Hadamard matrix, we present it directly in terms of generalized Hadamard matrices, avoiding the intermediate construction of the ARSRGD design.
\rev{The construction of the underlying group divisible design from a generalized Hadamard matrix is closely related to known constructions (see, e.g., \cite{KS2018,S2026}).
These constructions, however, are not concerned with pairwise additive BIB designs.
The purpose here is to use this structure to construct a family of incidence matrices satisfying the pairwise additivity condition.}
The following theorem shows that deleting the first column of a normalized generalized Hadamard matrix leads directly to full PABs.

\begin{thm}\label{thm:main2}
 The existence of a GH$(n,\lambda)$ \rev{with $\lambda\ge 2$ over an abelian group} implies the existence of a full PAB with parameters
 \begin{eqnarray*}
  v^\ast = n\lambda,\ b^\ast = n(n\lambda-1),\ r^\ast =n\lambda-1,\  k^\ast = \lambda,\ \lambda^\ast = \lambda-1.
 \end{eqnarray*}
\end{thm}

\begin{proof}
 Let $\vect A=(a_{ij})$ be a normalized GH($n,\lambda$) \rev{over} an abelian group $(G,+)$ with $G=\{x_1,x_2,\ldots,x_{n}\}$.
Further let $\vect A_h = (a_{ij}^{(h)})$ with $a_{ij}^{(h)} = a_{ij} + x_h$ ($1\le i,j\le n\lambda, 1\le h\le n$) be the $n\lambda \times n\lambda$ matrix.
Then, each $\vect A_h$ is also a GH($n,\lambda$). 

Now, we construct matrices $\vect N_h^\ast$ from the generalized Hadamard matrices $\vect A_h$.
Let $\vect A_h^\ast$ be an $n\lambda \times (n\lambda-1)$ matrix obtained by removing the first column of $\vect A_h$.
Then, replacing each entry $a_{ij}^{(h)} = x_{g}\in G$ of $\vect A_h^\ast$ by the $g$-th row vector of the identity matrix $\vect I_n$, we obtain $n\lambda \times n(n\lambda-1)$ matrices $\vect N_h^\ast$.
We show that the matrices $\vect N_h^\ast$ form the required full PAB.

By Lemma \ref{lem:GHproperty}, each column of $\vect A_h^\ast$ contains every element of $G$ exactly $\lambda$ times.
Hence, each column of $\vect N_h^\ast$ contains exactly $\lambda$ ones.
Since $\vect A$ is normalized, $\vert \{ j \mid a_{i_1j}^{(h)} = a_{i_2j}^{(h)}, 2\le j\le n\lambda \}\vert = \lambda -1$ holds for each $1\le i_1<i_2\le n\lambda$.
Hence, each $\vect N_h^\ast$ is an incidence matrix of a BIBD$(n\lambda,n(n\lambda-1), n\lambda-1, \lambda,\lambda-1)$.
Since $\{a_{ij}+x_h \mid 1\le h\le n\}=G$, we have $\sum_{h=1}^n \vect N_h^\ast = \vect J_{n\lambda\times n(n\lambda -1)}$.

On the other hand, 
$\vert \{ j \mid a_{i_1j}^{(h_1)} = a_{i_2j}^{(h_2)}, 2\le j\le n\lambda\}\vert 
= \vert \{ j \mid a_{i_1j}+x_{h_1} = a_{i_2j} +x_{h_2}, 2\le j\le n\lambda\}\vert 
= \vert \{ j \mid a_{i_1j} - a_{i_2j}= x_{h_2}-x_{h_1}, 2\le j\le n\lambda\}\vert
= \lambda$ 
for any $1\le i_1<i_2\le n\lambda$ and $1\le h_1 < h_2\le n$.
Since $x_{h_1}\neq x_{h_2}$ for $h_1\neq h_2$, we have $a_{ij}^{(h_1)}\neq a_{ij}^{(h_2)}$.
Therefore, the diagonal entries of $\vect N_{h_1}^\ast  (\vect N_{h_2}^\ast)^T$ are zero, while all off-diagonal entries are $\lambda$.
Thus, 
\begin{eqnarray}\label{eq:main2}
  \vect N_{h_1}^\ast  (\vect N_{h_2}^\ast)^T = \lambda(\vect J_{n\lambda}-\vect I_{n\lambda})
 \end{eqnarray}
for $h_1 \neq h_2$.
Equation (\ref{eq:main2}) implies (\ref{eq:N_1N_2}), and therefore the matrices $\vect N_h^\ast$ form the required full PAB.
The proof is complete.
\end{proof}

\rev{
Theorem~\ref{thm:main2} applies to any generalized Hadamard matrix over an abelian group satisfying the stated conditions; Lemma~\ref{lem:GH} provides only particular families of such matrices.
}
\rev{In particular, by Lemma~\ref{lem:GH} and Theorem~\ref{thm:main2},
the substitutions $(n,\lambda)=(5,4), (7,2)$, and $(9,2)$ give 
full $\mathrm{PAB}(20,4,3), \mathrm{PAB}(14,2,1)$, and $\mathrm{PAB}(18,2,1)$, respectively,
which correspond to Nos.~7, 11, and 15 in Table~\ref{table}.
More generally,} 
Theorem \ref{thm:main2} with a GH($q,2$) and a GH($q,4$) obtained by Lemma \ref{lem:GH} yields infinite families of full PABs for any prime or prime power $q\ge 2$ as follows.
\begin{Cor}\label{thm:main3}
 There exist a full PAB$(2q,2,1)$ and a full PAB$(4q,4,3)$ for any prime or prime power $q\ge 2$.
\end{Cor}

\rev{
The first family includes the previously known full PAB$(2^n,2,1)$ given in \cite{SMMKK2007}, as well as the
full PAB$(10,2,1)$ constructed in \cite{MK2013-1}.
For the second family, the case $q=4$, namely a full PAB$(16,4,3)$, was already known in \cite{SMMKK2007}.
To the best of the authors' knowledge, the remaining cases in Corollary 4.4 have not previously been reported.
}

The following example illustrates Theorem \ref{thm:main2} using a normalized 
\rev{GH$(7,2)$ obtained from the construction of Beth et al.\ \cite[Theorem 3.14]{BJL1999} by taking $q=7$ and choosing the nonsquare $3\in \mbox{GF}(7)$.}

\begin{exam}\label{exam:10}\rm
 \rev{
 The following array $A$ is a GH$(7,2)$ over $\mathbb Z_7$.
 \begin{align*}
  \vect A=
 \left(
\begin{array}{cccccccccccccc}
0&0&0&0&0&0&0&0&0&0&0&0&0&0\\
0&1&2&3&4&5&6&4&5&6&0&1&2&3\\
0&2&4&6&1&3&5&2&4&6&1&3&5&0\\
0&3&6&2&5&1&4&1&4&0&3&6&2&5\\
0&4&1&5&2&6&3&1&5&2&6&3&0&4\\
0&5&3&1&6&4&2&2&0&5&3&1&6&4\\
0&6&5&4&3&2&1&4&3&2&1&0&6&5\\
0&6&3&5&5&3&6&0&2&1&4&4&1&2\\
0&0&5&1&2&1&5&6&6&3&4&2&4&3\\
0&1&0&4&6&6&4&3&1&3&2&5&5&2\\
0&2&2&0&3&4&3&5&1&1&5&6&4&6\\
0&3&4&3&0&2&2&5&6&4&6&5&1&1\\
0&4&6&6&4&0&1&3&2&5&5&2&3&1\\
0&5&1&2&1&5&0&6&3&4&2&4&3&6
\end{array}
\right).
 \end{align*}
 Deleting the first column of $\vect A$ gives $\vect A^\ast$, from which a full PAB$(14,2,1$) is obtained, whose incidence matrices $\vect N_h^\ast$ ($1\le h\le 7$) are $14\times 91$ matrices.
}

\rev{
\begin{align*}
\setlength{\arraycolsep}{2pt}
  \vect N_h^\ast=
  \left(\begin{array}{ccccccccccccc}
  \vect v_{0+h} & \vect v_{0+h} & \vect v_{0+h} & \vect v_{0+h} & \vect v_{0+h} & \vect v_{0+h} & \vect v_{0+h} & \vect v_{0+h} & \vect v_{0+h} & \vect v_{0+h} & \vect v_{0+h} & \vect v_{0+h} & \vect v_{0+h} \\
  \vect v_{1+h} & \vect v_{2+h} & \vect v_{3+h} & \vect v_{4+h} & \vect v_{5+h} & \vect v_{6+h} & \vect v_{4+h} & \vect v_{5+h} & \vect v_{6+h} & \vect v_{0+h} & \vect v_{1+h} & \vect v_{2+h} & \vect v_{3+h} \\
  \vect v_{2+h} & \vect v_{4+h} & \vect v_{6+h} & \vect v_{1+h} & \vect v_{3+h} & \vect v_{5+h} & \vect v_{2+h} & \vect v_{4+h} & \vect v_{6+h} & \vect v_{1+h} & \vect v_{3+h} & \vect v_{5+h} & \vect v_{0+h} \\
  \vect v_{3+h} & \vect v_{6+h} & \vect v_{2+h} & \vect v_{5+h} & \vect v_{1+h} & \vect v_{4+h} & \vect v_{1+h} & \vect v_{4+h} & \vect v_{0+h} & \vect v_{3+h} & \vect v_{6+h} & \vect v_{2+h} & \vect v_{5+h} \\
  \vect v_{4+h} & \vect v_{1+h} & \vect v_{5+h} & \vect v_{2+h} & \vect v_{6+h} & \vect v_{3+h} & \vect v_{1+h} & \vect v_{5+h} & \vect v_{2+h} & \vect v_{6+h} & \vect v_{3+h} & \vect v_{0+h} & \vect v_{4+h} \\
  \vect v_{5+h} & \vect v_{3+h} & \vect v_{1+h} & \vect v_{6+h} & \vect v_{4+h} & \vect v_{2+h} & \vect v_{2+h} & \vect v_{0+h} & \vect v_{5+h} & \vect v_{3+h} & \vect v_{1+h} & \vect v_{6+h} & \vect v_{4+h} \\
  \vect v_{6+h} & \vect v_{5+h} & \vect v_{4+h} & \vect v_{3+h} & \vect v_{2+h} & \vect v_{1+h} & \vect v_{4+h} & \vect v_{3+h} & \vect v_{2+h} & \vect v_{1+h} & \vect v_{0+h} & \vect v_{6+h} & \vect v_{5+h} \\
  \vect v_{6+h} & \vect v_{3+h} & \vect v_{5+h} & \vect v_{5+h} & \vect v_{3+h} & \vect v_{6+h} & \vect v_{0+h} & \vect v_{2+h} & \vect v_{1+h} & \vect v_{4+h} & \vect v_{4+h} & \vect v_{1+h} & \vect v_{2+h} \\
  \vect v_{0+h} & \vect v_{5+h} & \vect v_{1+h} & \vect v_{2+h} & \vect v_{1+h} & \vect v_{5+h} & \vect v_{6+h} & \vect v_{6+h} & \vect v_{3+h} & \vect v_{4+h} & \vect v_{2+h} & \vect v_{4+h} & \vect v_{3+h} \\
  \vect v_{1+h} & \vect v_{0+h} & \vect v_{4+h} & \vect v_{6+h} & \vect v_{6+h} & \vect v_{4+h} & \vect v_{3+h} & \vect v_{1+h} & \vect v_{3+h} & \vect v_{2+h} & \vect v_{5+h} & \vect v_{5+h} & \vect v_{2+h} \\
  \vect v_{2+h} & \vect v_{2+h} & \vect v_{0+h} & \vect v_{3+h} & \vect v_{4+h} & \vect v_{3+h} & \vect v_{5+h} & \vect v_{1+h} & \vect v_{1+h} & \vect v_{5+h} & \vect v_{6+h} & \vect v_{4+h} & \vect v_{6+h} \\
  \vect v_{3+h} & \vect v_{4+h} & \vect v_{3+h} & \vect v_{0+h} & \vect v_{2+h} & \vect v_{2+h} & \vect v_{5+h} & \vect v_{6+h} & \vect v_{4+h} & \vect v_{6+h} & \vect v_{5+h} & \vect v_{1+h} & \vect v_{1+h} \\
  \vect v_{4+h} & \vect v_{6+h} & \vect v_{6+h} & \vect v_{4+h} & \vect v_{0+h} & \vect v_{1+h} & \vect v_{3+h} & \vect v_{2+h} & \vect v_{5+h} & \vect v_{5+h} & \vect v_{2+h} & \vect v_{3+h} & \vect v_{1+h} \\
  \vect v_{5+h} & \vect v_{1+h} & \vect v_{2+h} & \vect v_{1+h} & \vect v_{5+h} & \vect v_{0+h} & \vect v_{6+h} & \vect v_{3+h} & \vect v_{4+h} & \vect v_{2+h} & \vect v_{4+h} & \vect v_{3+h} & \vect v_{6+h}
  \end{array}\right),
\end{align*}
where $\vect v_0 = (1,0,0,0,0,0,0)$, $\vect v_1 = (0,1,0,0,0,0,0)$, $\vect v_2 = (0,0,1,0,0,0,0)$, $\vect v_3 = (0,0,0,1,0,0,0)$, $\vect v_4 = (0,0,0,0,1,0,0)$, $\vect v_5 = (0,0,0,0,0,1,0)$, $\vect v_6 = (0,0,0,0,0,0,1)$, and the subscripts are taken modulo $7$.
}
\end{exam}

\section{Concluding remarks}\label{sec:conclusion}

We summarize the full PABs obtained by our construction.
First, a full PAB$(20,4,3)$, a full PAB$(14,2,1)$, 
and a full PAB$(18,2,1)$ obtained by Theorem \ref{thm:main2} establish the existence of full PABs for parameters (Nos.\,7, 11, and 15 in Table \ref{table} in Section \ref{sec:introduction}) whose existence was previously unknown in \cite{SMMKK2007}.
 
In addition to Lemma \ref{lem:GH}, many classes of generalized Hadamard matrices are known in the literature (e.g., \cite{HSS1999, L2007}).
Hence, our method can be expected to establish further existence results.
In particular, the resulting designs attain the lower bound in (\ref{eq:lambda}) if $\lambda$ in GH($n,\lambda$) is even.
Unfortunately, our approach cannot provide any full PAB with odd $k$ and $\lambda = (k-1)/2$.
Moreover, many full PABs remain unknown, even in the case $\lambda = k-1$.
For example, even the existence of a full PAB$(12,2,1)$ remains unknown.

Note that each of the conditions (\ref{eq:main}) and (\ref{eq:main2}) is stronger than (\ref{eq:N_1N_2}), since each immediately implies (\ref{eq:N_1N_2}), whereas the converse does not hold. 
Therefore, the designs constructed in Theorems \ref{thm:main} and \ref{thm:main2} form a special class of PABs.
\rev{In particular, the infinite families of full PABs in Corollary \ref{thm:main3} satisfy the stronger condition (\ref{eq:main2}).}
This class can be studied within the framework of symmetric balanced nested designs introduced in \cite{FKKMS2002}.

On the other hand, PABs are closely related to combinatorial multi-arrays, called perpendicular multi-arrays and ordered multi-designs (see \cite{LLDC2018,MK2021,MK2022}).
\rev{The constructions presented in this paper may provide new existence results for these combinatorial multi-arrays.
A detailed investigation of these connections, including the corresponding parameter mappings and existence problems, as well as further constructions of full PABs, is left for future work.}

\section*{Author Contributions}

The study was motivated and shaped by earlier discussions between Kazuki
Matsubara and Satoru Kadowaki. Kazuki Matsubara developed the construction,
obtained the main results, and prepared the manuscript.

\section*{Funding}

This work was partially supported by JSPS KAKENHI Grant Number JP23K02359.

\section*{Data Availability Statement}

No new data were generated or analyzed in this study. Data sharing is not applicable to this article.

\section*{Acknowledgments}

The authors sincerely thank the late Professor Sanpei Kageyama for valuable discussions that helped lay the foundation for this research.
\rev{The authors also thank the referee for valuable comments and suggestions that helped improve the manuscript.}

\section*{Conflicts of Interest}

The authors declare no conflicts of interest.

\section*{Declaration of Generative AI and AI-Assisted Technologies in the Writing Process}

During the preparation of this work, the authors used ChatGPT (OpenAI) for
language editing and suggestions to improve the clarity and organization of
the manuscript. After using this tool, the authors reviewed and edited the
content as needed and take full responsibility for the content of the
published article.


\end{document}